\documentclass[12pt]{article}
\usepackage[utf8]{inputenc}
\usepackage[left=1in, right=1in, top=1in,bottom = 2in]{geometry}
\usepackage[shortlabels]{enumitem}
\usepackage{amsmath}
\usepackage{amsthm}
\usepackage{amssymb}
\usepackage{amsfonts}
\usepackage{array}
\usepackage{color}
\usepackage{comment}
\usepackage{dsfont}
\usepackage{easytable}
\usepackage{enumitem}
\usepackage{environ}
\usepackage{fancyhdr}
\usepackage{framed}
\usepackage{graphicx}
\usepackage{hyperref}
\usepackage{listings}
\usepackage{multirow}
\usepackage{tikz-cd}
\usepackage{wasysym}
\usepackage{xcolor}

\newtheorem{theorem}{Theorem}

\newtheorem{lemma}[theorem]{Lemma}
\newtheorem{definition}[theorem]{Definition}

\usepackage{arydshln}
\makeatletter
  \renewcommand*\env@matrix[1][*\c@MaxMatrixCols c]{%
    \hskip -\arraycolsep
    \let\@ifnextchar\new@ifnextchar
  \array{#1}}
\makeatother

\newcommand{\Z}{\mathbb{Z}}

\newcommand{\F}{\mathbb{F}}

\newcommand{\Tr}{\mbox{Tr}}

\providecommand{\keywords}[1]
{
  \small	
  \textbf{\textit{Keywords---}} #1
}

\providecommand{\subjclass}[1]
{
  \small	
  \textbf{\textit{AMS Subjclass---}} #1
}

\begin{document}
\author{
  José Gustavo Coelho\\
}
\title{Solutions of full equations related to diagonal equations}
\maketitle
\begin{abstract}
Let $p$ be a prime number, $m$ be an even positive integer, and $\F_q$ be a finite field with $q = p^m$ elements. In this paper, we compute the number of solutions with all coordinates in $\F_q^*$ for diagonal equations of the form
$$a_1 x_1^{d} + \dots + a_s x_s^{d} = b, \quad a_i \in \F_q^*, \, b \in \F_q,$$
when the coefficients and exponents satisfy specific arithmetic conditions that facilitate the computation through pure Gauss sums. We then apply this result to determine the number of solutions for equations of the form
\begin{equation*}
    a_1 x_1^{d_{1,1}} \cdots x_n^{d_{n,1}} + \dots + a_s x_1^{d_{1,s}}\cdots x_n^{d_{n,s}} = b, 
\end{equation*}
where all exponents are positive, and the equation is related in a particular way to diagonal equations with the aforementioned characteristics.
\end{abstract}



\section{Introduction}

Finding the number of solutions to equations over a fixed finite field is an ongoing question that has received considerable attention. Geometrically, this problem translates to determining the affine points of hypersurfaces defined over finite fields. While geometric bounds are established for the number of affine points\cite{argentinas}\cite[Appendix C]{hartshorne}, determining the specific quantity is a difficult problem, which is done according to the type of equation.

Let $\F_q$ denote a finite field with $q$ elements. For reasons that will become clear in the following sections, the number of solutions can be explicitly computed for many families of diagonal equations, i.e. equations of the form
$$a_1 x_1^{d_1} + a_2 x_2^{d_2} + \cdots + a_s x_s^{d_s} = b,$$
where $a_i \in \F_q^*$, $d_i \in \Z_{>0}$ for all $1 \le i \le s$ and $b \in \F_q$. In particular, the case where $d_i = 1$ for all $1 \le i \le s$ is trivial, and the case where $d_i \in \{1,\, 2\}$ for all $1 \le i \le s$ has been addressed \cite[Chapter 10]{gauss jacobi}. In general, the number of solutions to diagonal equations can be determined under the assumption that the exponents satisfy specific properties \cite{artigo do jose}.


The number of solutions for various families of equations has been extensively studied.
For instance, Carlitz \cite{carlitz} and Baoulina \cite{baoulina} studied specific cases of equations of the form
\begin{equation}
  x_1^{d_1} + x_2^{d_2} + \dots + x_s^{d_s} = b x_1 \cdots x_s,
\end{equation}
where $d_j\in \Z_{> 0}$, $a_j\in \F_q^*$ for all $1 \le j \le s$ and $b\in \F_q^*$. These cases are subcases of the Markoff-Hurwitz equations, i.e., those of the form 
\begin{equation}\label{markoff}
  x_1^{m_1} + x_2^{m_2} + \dots + x_n^{m_n} = b x_1^{t_1} x_2^{t_1} \dots x_n^{t_n}.
\end{equation}
where $m_j, t_j>0$ for all $j=1,\dots, n$ and $b\in \F_q^*$. The exploration of this family would be further expanded upon in many other cases. Notably, Cao \cite{cao 2010} calculated the number of solutions over finite fields of Markoff-Hurwitz equations in the case where $\gcd (m\sum_{i=1}^n t_i/m_i -m, q-1)=1$ and $m = m_1 m_2 \cdots m_n$, while the case where $\gcd (m\sum_{i=1}^n t_i/m_i -m, q-1)>1$ was considered by Cao, Jiang, and Gao \cite{wei cao}. In the case $m_1=m_2=\dots=m_n$ and $t_1=\dots=t_n=1$, these equations define hypersurfaces known as Calabi-Yau's hypersurfaces which have been studied extensively \cite{calabi yau 1, calabi yau 2}. Other results regarding the number of solutions for the general equation (\ref{markoff}) are available in the literature.  

Most of the families mentioned are sparse, as the majority of the terms involve only one variable, not straying too far from diagonal equations. For equations that are neither diagonal nor sparse, direct computation can prove difficult. One case where the equation is not sparse but computation is viable are the full equations, which include every variable in every term, except possibly for a constant term. Cao, Wen, and Wang \cite{cao 2017} have determined the number of solutions for full equations of the form
\begin{equation*}
    a_1 x_1^{d_{1,1}} \cdots x_n^{d_{1,n}} + \dots + a_s x_1^{d_{s,1}}\cdots x_n^{d_{s,n}} = 0,
\end{equation*}
where $d_{i,j} > 0$ for all $i, j$, assuming that the matrix $(d_{i,j})_{i,j}$ is row-equivalent to a diagonal matrix $D$, when the elements in the diagonal of $D$ are only 1's and 2's.

In a similar vein, a previous work by the author of the present text investigated the number of roots of polynomials of the form
$$f(x_1, \dots, x_n) = - b + a_1 x_1^{d_{1,1}} + a_2 x_1^{d_{2,1}} x_2^{d_{2,2}}+ a_3 x_1^{d_{3,1}} x_2^{d_{3,2}} x_3^{d_{3,3}}+ \cdots + a_n x_1^{d_{n,1}} \cdots x_n^{d_{n,n}},$$
when they are related in a certain arithmetic way to diagonal polynomials that are either linear or quadratic \cite{meu artigo}.

In this text, we first derive an expression for the number of $*$-solutions, i.e., solutions where no coordinate is $0$, of diagonal equations of the form
$$a_1 x_1^{d} + \cdots + a_n x_n^{d} = b,$$
where $a_1, \dots, a_n \in \F_q^*$, $b \in \F_q$, $0 \le d \le q - 1$, and these parameters satisfy certain arithmetic conditions. As an application of this result, we count the total number of solutions for full equations that are related in a specific way to the aforementioned diagonal equations.

\section{Preliminaries}
Let $p$ be a prime number, $q$ a power of $p$, and $\alpha$ a fixed primitive element of $\F_q$. For each integer $d$, we define $\zeta_d = \exp(2 \pi I /d)$, where $I$ is the imaginary unit, as a primitive $d$-th root of unity in $\mathbb{C}$. 

Let $\Tr$ be the trace function from $\F_q$ to $\F_p$. We define the canonical additive character $\psi: \F_q \rightarrow \mathbb{C}$ as
\begin{equation}\label{caracter aditivo canonico}
    \psi(x) = \zeta_p^{\Tr(x)}. 
\end{equation}

For $d \mid (q - 1)$ we define a multiplicative character $\eta_d:\F_q^* \rightarrow \mathbb{C}$ with order $d$ as
$$\eta_d (x) =\begin{cases}
    \zeta_d^{\log(x)},& \text{ if } x \ne 0, \\
    0, & \text{ if } x = 0,\\
\end{cases}$$
where the discrete logarithm $\log(x)$ is defined as the least positive integer such that $\alpha^{\log(x)} = x$. We note that $\log: \F_q^{*} \rightarrow \Z_{q-1}$ depends on the fixed primitive element $\alpha$.

We implicitly extend the domain of each non-trivial multiplicative character $\eta$ to $\F_q$ by setting $\eta (0) = 0$, and the trivial multiplicative character $\eta_0$ is defined as $\eta_0 (x) = 1$ for all $x \in \F_q$. For any character $\eta$, additive or multiplicative, we define its conjugate character as $\overline{\eta}(x) = \overline{\eta(x)}$, where the bar denotes complex conjugation.
The following lemma will be useful:

\begin{lemma}\label{lema eta_d}
    Let $x \in \F_q$, then
    $$\sum_{j = 0}^{d - 1} \eta_d^j(x) = \begin{cases}
        d & \text{ if } \eta_d(x) = 1,\\
        0 & \text{ otherwise.}
    \end{cases}$$
\end{lemma}
\begin{proof}
    If $\eta_d(x) = 1$, the equality is straightforward. If $\eta_d(x) \ne 1$, the expression is a geometric sum that vanishes.
\end{proof}

For each multiplicative character $\eta$, we define its Gauss sum as
\begin{equation}\label{definicao soma de gauss}
    G(\eta) = \sum_{c \in \F_q^*} \psi(c) \eta(c).
\end{equation}
There are some cases where the value of the Gauss sum can be straightforwardly computed. The following definition provides one such case:
\begin{definition}\label{pure sum}
    Let $m$ be an even positive integer. An integer $d \ge 3$ such that $d \mid q - 1$ is called $(p, r)$-admissible if there is a positive integer $l$ such that $2l \mid m$ and $d \mid p^l + 1$, and $r$ is the smallest positive integer satisfying that. When $d$ is $(p, r)$-admissible for some integer $r$, the Gauss sums $G(\eta_d^j)$ for $1 \le j \le d - 1$ are called pure Gauss sums. 
\end{definition}

We remark that this definition requires the exponent $m$ to be even. Thus, when using pure sums, we will always assume that $q$ is a square. The following result allows us to compute pure sums:

\begin{lemma}\label{Lema da soma pura}
    Let $d \ge 3$ be $(p, r)$-admissible. Then,
    \begin{equation*}
        G(\eta_d^j) = \begin{cases}
            q^{1/2} (-1)^{jh + h + 1}, & \text{ if }2 \mid d \text{ and } 2\nmid (p^r + 1)/d,\\
            q^{1/2} (-1)^{h + 1}, & \text{ if } 2 \nmid d \text{ or } 2 \mid (p^r + 1) /d,
        \end{cases}
    \end{equation*}
    where $h = m/(2r)$.
\end{lemma}
\begin{proof}
    See Theorem 1 in \cite{pure gauss sums}
\end{proof}

Let us establish some notation for monomials that will allow us to express the results regarding characters and polynomials in a concise manner. Let $\F_q[x_1, \dots, x_n]$ denote the polynomial ring over $\F_q$ with $n$ variables. For any $D=(d_1, \dots, d_n) \in \Z_{\ge 0}^n$ we will use the notation $X^D = x_1^{d_1} \cdots x_n^{d_n}$. For each equation
$$\sum_i a_i X^{D_i} = \sum_j b_j X^{D_j},$$
we define its corresponding polynomial $f(X) \in \F_q[x_1, \dots, x_n]$ as
$$f(X) = \sum_i a_i X^{D_i} - \sum_j b_j X^{D_j}.$$
This formulation transforms the problem of determining the number of solutions of an equation into the task of counting the number of roots of a polynomial. We will exclusively use polynomial notation from now on.

Given a polynomial $f \in \F_q[x_1, \dots, x_n]$ of the form
\begin{equation}\label{f generico}
    f(x_1, \dots, x_n) = \sum_{j=1}^s a_j X^{D_j},
\end{equation}
where $D_j = (d_{1,j}, \dots, d_{n, j}) \in \Z_{\ge 0}^n$ and $a_j \ne 0$ for all $j = 1, \dots, s$, we define $N(f)$ as the number of roots of $f(x_1, \dots,$ $ x_n)$ over $\F_q^n$ and $N^*(f)$ as the number of roots over $(\F_q^*)^n$, which we will denote as $*$-roots of the polynomial.
Let us define the degree matrix of $f$ as $D_f = (D_1^T, \dots, D_s^T)$, where the notation $D^T$ indicates that the vector $D$ is being transposed into a column of the matrix. The augmented degree matrix of $f$ as $\tilde{D}_f = ((\tilde{D_1})^T, \dots, (\tilde{D_s})^T)$, where $(\tilde{D_j}) = (1, D_j)$.

Besides diagonal polynomials, another family of polynomials that will be important in the text are the full polynomials, which we define as polynomials of the form
\begin{equation}\label{polinomio cheio}
    f(X) = a_1 x_1^{d_{1,1}} \cdots x_n^{d_{n,1}} + \dots + a_s x_1^{d_{1,s}}\cdots x_n^{d_{n,s}} = b,
\end{equation}
where $a_1, \dots, a_s \in \F_q^*$, $b \in \F_q$, and every exponent $d_{i,j}$ in the degree matrix is greater than zero.
\newline
\hfill

The following lemma is the basis for the approach we will take for counting roots of polynomials with characters:
\begin{lemma}\label{lema base contagem}
    Let $\psi$ be the canonical additive character over $\F_q$ and $x \in \F_q$. Then,
    $$\frac{1}{q} \sum_{c \in \F_q} \psi(c x) = \begin{cases}
        1, & \text{ if }x = 0,\\
        0, & \text{ if }x \ne 0.
    \end{cases}$$
\end{lemma}
\begin{proof}
    When $x = 0$, we have $\psi(c x) = 1$ for every $c \in \F_q$, and the result follows. When $x \ne 0$, it is a direct consequence of Theorem 5.4 in \cite{lidl finite fields}.
\end{proof}
As a consequence of this result, character sums of this form can be used as indicator functions to find roots of polynomials.
Let $f$ be a polynomial of the form (\ref{f generico}), using Lemma \ref{lema base contagem}, we have that the number of roots over $(\F_q^*)^n$ is 
\begin{equation}\label{usando caracteres generico}
\begin{split}
    N^*(f) &= \sum_{x_1, \dots, x_n \in \F_q^*} \frac{1}{q} \sum_{c \in \F_q} \psi(c f(X))\\
    &= \frac{1}{q} \sum_{c \in \F_q} \sum_{x_1, \dots, x_n \in \F_q^*} \psi(c f(X)).
\end{split}
\end{equation}
Let us define
\begin{equation}\label{def S(u, d)}
    S(u, d) = \sum_{x \in \F_q^*} \psi\left(u x^{d}\right).
\end{equation}
If we require the polynomial to be diagonal, we can rewrite the expression in \eqref{usando caracteres generico} in terms of sums of the form \eqref{def S(u, d)}:
\begin{lemma}\label{lema da contagem}
Let $f$ be a diagonal polynomial of the form
$$f(X) = a_1 x_1^{d_1} + \cdots + a_s x_1^{d_s} - b,$$ 
where $a_1, \dots, a_s \in \F_q^*$ and $b \in \F_q$. Then
\begin{equation}\label{N*(g) com S(c a_i, d_i) generico}
    N^*(f) = \frac{(q - 1)^s}{q} + \frac{1}{q}\sum_{c \in \F_q^*} \overline{\psi}(cb) \prod_{j = 1}^{s} S(ca_j, d_j).
\end{equation}
\end{lemma}
\begin{proof}
    The expression in (\ref{usando caracteres generico}) simplifies to
\begin{align*}
    N^*(f) &= \frac{1}{q} \sum_{c \in \F_q} \sum_{x_1, \dots, x_s \in \F_q^*} \psi\left(c \left(- b + \sum_{j = 1}^s a_j x_j^{d_j}\right)\right)\\
    &= \frac{(q - 1)^s}{q} + \frac{1}{q}\sum_{c \in \F_q^*} \overline{\psi}(cb) \prod_{j = 1}^{n}\sum_{x \in \F_q^*} \psi\left(c a_j x^{d_j}\right).
\end{align*}
\end{proof}

Lemma \ref{lema da contagem} can be used to explicitly compute the number of $*$-roots of a polynomial, provided the Gaussian sums can be computed. In the main results section this Lemma will be used on a family of diagonal polynomials that satisfies specific arithmetic properties.
\newline
\hfill

If we use multiplicative characters instead of additive characters, a similar result for counting $*$-roots in terms of multiplicative characters and Gauss sums can be derived. Let $\omega$ be a generator of the multiplicative characters group, i.e., $\widehat{\F_q^*} = \{\omega^k: k = 0, 1, \dots, q - 2\}$. 
The following result allows us to express $N^*(f)$ in terms of $\omega$ and the Gauss sums.

\begin{lemma}\label{lema do caractere}
Let $f$ be a polynomial of the form (\ref{f generico}), then
$$N^*(f) = \frac{(q - 1)^n}{q} + \frac{(q - 1)^{n+1-s}}{q}\sum \prod_{j=1}^s \omega(a_j)^{v_j} G(\omega^{-v_j}),$$
where the sum is taken over all vectors $v = (v_1, \dots, v_s)$ with $0 \le v_j \le q - 2$ for $j = 1, \dots, s$ such that $\tilde{D}_fv^T \equiv 0 \pmod{q - 1}$.
\end{lemma}
\begin{proof}
See Lemma 2.5 in \cite{wei cao}.
\end{proof}

This lemma is not ideal for explicit computations. Since the sum is computed over many vectors that are solutions of a linear system over a ring, it cannot be simplified further unless the set of solutions is known and constrained in some way. However, it has a significant property: if the terms in the expression evaluate identically for two different polynomials, it follows that the two polynomials must have the same number of $*$-roots. To formalize this concept, we will introduce the following definition:

\begin{definition}
    Two polynomials $f = \sum_{j=1}^m a_j X^{D_j}$ and $g = \sum_{j =1 }^m a_j X^{D_{j}^\prime}$ are said to be $*$-equivalent if they have the same coefficient vector $(a_1, \dots, a_m)$ and the congruences $\tilde{D}_f v^T \equiv 0 \pmod{q - 1}$ and $\tilde{D}_g v^T \equiv 0 \pmod{q - 1}$ have the same set of solutions.
\end{definition}

It follows from Lemma \ref{lema do caractere} that if $f$ and $g$ are $*$-equivalent polynomials, then $N^*(f) = N^*(g)$, i.e., they have the same number of $*$-roots.

Determining if two polynomials $f$ and $g$ are $*$-equivalent can be challenging, as it requires verifying whether the linear systems $\tilde{D}_f v^T = 0$ and $\tilde{D}_g v^T = 0$ have the same set of solutions. It can be verified that two matrices $D$ and $E$ with coefficients in $\Z_{q-1}$ satisfying $D v^T \equiv 0 \pmod{q - 1}$ and $E v^T \equiv 0 \pmod{q - 1}$ have the same set of solutions if there exists an invertible matrix $M$ over $\Z_{q-1}$ such that $M D = E$. In this case, $D$ and $E$ are said to be row-equivalent. 

Hence, if two polynomials $f$ and $g$ have the same coefficient vector and there is an invertible matrix $M$ over $\Z_{q-1}$ such that $M \tilde{D}_f = \tilde{D}_g$, then $f$ and $g$ are $*$-equivalent. This condition is sufficient but not strictly necessary, since row-equivalence over $\Z_{q-1}$ is not the only way for the systems $D v^T \equiv 0 \pmod{q - 1}$ and $E v^T \equiv 0 \pmod{q - 1}$ to have the same solution set.

In particular, the elementary row operations can be represented as multiplication by invertible matrices. Thus, if we can apply a sequence of elementary row operations to transform the matrix $\tilde{D}_f$ into $\tilde{D}_g$, then the corresponding linear systems have the same set of solutions. This provides a sufficient criterion for establishing $*$-equivalency. The elementary row operations are:
\begin{enumerate}[$(i)$]
    \item swapping two rows;

    \item adding a multiple of a row to another;

    \item multiplying a row by an element in $\Z_{q-1}^*$;
\end{enumerate}
which can be represented by multiplying by invertible matrices.

To address possible misconceptions, we remark that even though $N^*(f) = N^*(g)$ for two $*$-equivalent polynomials $f$ and $g$, this does not mean they have the same set of $*$-roots. For instance, the polynomials $f(x,y) = x^2y^3 + x y^2 $ and $g(x,y) = xy + x^3 y^2$ in $\F_5[x,y]$ are $*$-equivalent, but it can be verified that $(2,2)$ is a root of $f$ but not a root of $g$.

We also remark that two polynomials being $*$-equivalent does not guarantee that they have the same number of roots in $\F_q^n$. For instance, the polynomials $f(x,y,z) = 11x^{13} + 5 x^{21} y^{19} + 12 x^{2}y^{3}z^{17}$ and $g(x,y,z) = 11 x + 5 y + 12 z$  are $*$-equivalent in $\F_{31}[x,y,z]$, thus $N^*(f) = 870 = N^*(g)$. However, it can be verified that $N(f) = 1861 \ne 961 = N(g)$.

\section{Main results}

In this section we will present two results. First, we will determine the number of $*$-roots for a family of diagonal polynomials. Then, we will determine the number of roots of full polynomials when they are $*$-equivalent to said diagonal polynomials.

Lemma \ref{lema do caractere} allows us to relate the number of $*$-roots of different polynomials. However, to utilize this result, we need a class of polynomials for which we can explicitly compute the number of $*$-roots. As mentioned previously, this is already established for linear and quadratic diagonal polynomials. In what follows, we will determine the number of $*$-roots for a different family of diagonal polynomials.

From this point onward, the results will rely on computing pure Gauss sums, so we will always assume that $q$ is a square. Lemma \ref{lema da contagem} suggests that being able to compute $S(c a_j, d_j)$ for $1 \le j \le s$, where $S(u, d)$ is as defined in \eqref{def S(u, d)}, would allow us to count the number of $*$-roots of a diagonal polynomial. The following result addresses this calculation when the exponents are $(p, r)$-admissible for a suitable positive integer $r$.
\begin{lemma}\label{Lema S(c a, d)}
    Let $u$ be an element in $\F_q^*$, $d \ge 3$ a $(p, r)$-admissible positive integer and $h = m/(2r)$. If $p$ is odd, we have
    $$S(u, d) = \begin{cases}
        - 1 - (-1)^{h} (d - 1) q^{1/2}, & \text{ if } \eta_d(u) = (-1)^{h(p^r+1)/d},\\
        - 1 + (-1)^{h} q^{1/2},& \text{ if } \eta_d(u) \ne (-1)^{h(p^r+1)/d}.
    \end{cases}$$
    If $p = 2$, then
    $$S(u, d) = \begin{cases}
        - 1 - (-1)^{h} (d - 1) q^{1/2}, & \text{ if } \eta_d(u) = 1,\\
        - 1 + (-1)^{h} q^{1/2},& \text{ if } \eta_d(u) \ne 1.
    \end{cases}$$
\end{lemma}
\begin{proof}
We use the Lemma \ref{lema eta_d} to rewrite the sum:
    \begin{align*}
        S(u, d) &= \sum_{x \in \F_q^*} \psi\left(u x^{d}\right) \\
        &=  \sum_{y \in \F_q^*} \psi(u y) \sum_{j = 0}^{d - 1} \eta_d^j(y)\\
        &=  \sum_{y \in \F_q^*} \psi(u y) \sum_{j = 0}^{d - 1} \eta_d^j(y) \cdot \eta_d^j(u) \overline{\eta_d}^j(u)\\
        &=  \sum_{j = 0}^{d - 1} \overline{\eta_d}^j(u) \sum_{y \in \F_q^*} \psi(u y) \eta_d^j(u y)\\
        &=  \sum_{j = 0}^{d - 1} \overline{\eta_d}^j(u) \sum_{y \in \F_q^*} \psi(y) \eta_d^j(y)\\
        &= -1 + \sum_{j = 1}^{d - 1} \overline{\eta_d}^j(u) G(\eta_d^j).
    \end{align*}
    Our assumption that $d$ is $(p, r)$-admissible guarantees that the Gauss sums are pure. Let us define
    \begin{equation}\label{definicao de delta}
        \Delta(d, j) = \begin{cases}
        (-1)^{jh}, &\text{ if }2 \mid d \text{ and } 2\nmid (p^r + 1)/d,\\
        1, & \text{ if }2 \nmid d \text{ or } 2\mid (p^r + 1)/d.\\
    \end{cases}
    \end{equation}
    Using Lemma \ref{Lema da soma pura} and Equation (\ref{definicao de delta}), we obtain
    \begin{equation}\label{S(ca, d) generico}
    \begin{split}
        S(u, d) &= -1  + \sum_{j = 0}^{d - 1} \overline{\eta_d}^j(u) (-1)^{h + 1} q^{1/2} \Delta(d, j)\\
        &= -1 + (-1)^{h + 1} q^{1/2} T(u,d),
    \end{split}
    \end{equation}
    where $T(u, d) = \sum_{j = 1}^{d - 1} \overline{\eta_d}^j(u) \Delta(d, j)$. 
    
    If $2 \mid d$ and $2 \nmid (p^r + 1)/d$, then $\Delta(d, j) = (-1)^{jh}$. Thus, in this case we have
    \begin{equation}\label{T(ca,d) 2 mid d}
        T(u, d) = \sum_{j=1}^{d-1} ((-1)^h \overline{\eta_d}(u))^j = \begin{cases}
        d - 1, &\text{ if } \eta_d(u) = (-1)^{h},\\
        -1, & \text{ if } \eta_d(u) \ne (-1)^{h}.
    \end{cases}
    \end{equation}
    Similarly, if $2 \nmid d$ or $2 \mid (p^r + 1)/d$ then $\Delta(d, j) = 1$, in which case
    \begin{equation}\label{T(ca,d) 2 nmid d}
    T(u, d) = \sum_{j=1}^{d-1} (\overline{\eta_d}(u))^j = \begin{cases}
        d - 1, &\text{ if } \eta_d(u) = 1,\\
        -1, & \text{ if } \eta_d(u) \ne 1.
    \end{cases}
    \end{equation}
    We will now verify that 
    \begin{equation}\label{casos que juntam os casos T(ca, d)}
        (-1)^{h(p^r + 1)/d} = \begin{cases}
        (-1)^h, & \text{ if } 2 \mid d \text{ and } 2\nmid (p^r + 1)/d,\\
     1, & \text{ if }2 \nmid d \text{ or } 2\mid (p^r + 1)/d.\\
    \end{cases}
    \end{equation}
    For the first case, we notice that the condition $2 \nmid (p^r + 1)/d$ implies $(-1)^{h (p^r + 1)/d} = (-1)^{h}$. Analogously, if $2 \mid (p^r + 1)/d$, then $(-1)^{h (p^r + 1)/d} = 1$. The condition $2 \nmid d$ on the second case also implies $(-1)^{h (p^r + 1)/d} = 1$, since $p$ being odd implies $2 \mid p^r + 1$ in that case. Combining \eqref{T(ca,d) 2 mid d}, \eqref{T(ca,d) 2 nmid d}, and \eqref{casos que juntam os casos T(ca, d)}, we have
    \begin{equation*}
        T(u, d) = \begin{cases}
        d - 1, &\text{ if } \eta_d(u) = (-1)^{h(p^r + 1)/d},\\
        -1, & \text{ if } \eta_d(u) \ne (-1)^{h(p^r + 1)/d},
    \end{cases}
    \end{equation*}
    This equation, when applied to (\ref{S(ca, d) generico}), gives us the desired formula for $S(u, d)$.
    

    Analogously, when $p = 2$ we have $d \mid p^r + 1$, and consequently $2 \nmid d$. Thus $T(u, d)$ is as given in (\ref{T(ca,d) 2 nmid d}), which when applied to (\ref{S(ca, d) generico}) gives us the desired result.
    
\end{proof}

For convenience, let us introduce the following definitions:
\begin{definition}\label{definicao A e B}
    Let $d\ge 3$ be a $(p, r)$-admissible integer and  $h = m/(2r)$. We define the functions $C_1$ and $C_2$ as
    \begin{align*}
    C_1(d) &= - 1 - (-1)^{h} (d - 1) q^{1/2},\\
    C_2(d) &= - 1 + (-1)^{h} q^{1/2}.
\end{align*}
\end{definition}
We remark that since $q$ is fixed and $r$ and $h$ are defined in terms of $d$, only $d$ is needed to compute these values.

\hfill
\newline

Applying Lemma \ref{Lema S(c a, d)} to Equation (\ref{N*(g) com S(c a_i, d_i) generico}) allows us to compute the number of $*$-roots for diagonal polynomials satisfying that special condition.
From now on, let us consider the family of diagonal polynomials $g \in \F_q[x_1, \dots, x_s]$ of the form
\begin{equation}\label{g todos ds iguais}
    g(X) = a_1 x_1^d + \cdots + a_s x_s^d - b,
\end{equation}
where $\eta_d(a_1) = \dots = \eta_d(a_s)$ for some integer $d \ge 3$ that is a $(p,r)$-admissible integer for a suitable positive integer $r$.

\begin{theorem}\label{teorema diagonal caso b = 0}
    Let $g \in \F_q[x_1, \dots, x_s]$ be a polynomial of the form (\ref{g todos ds iguais}) where $b = 0$. Then
    $$N^*(g)  = \frac{1}{q} \left((q - 1)^{s}+ \frac{q-1}{d} C_1(d)^s + \frac{(q - 1)(d - 1)}{d} C_2(d)^s\right).$$
\end{theorem}
\begin{proof}
    From (\ref{N*(g) com S(c a_i, d_i) generico}) we have
    $$N^*(g) = \frac{(q - 1)^s}{q} + \frac{1}{q}\sum_{c \in \F_q^*} \prod_{i = 1}^{s} S(ca_i, d).$$
    Let us suppose $p$ is odd and $(-1)^{h(p^r + 1)/d} = 1$ so the condition we need to check to determine the value of $S(c a_i, d)$ is whether $\eta_d(c a_i)$ equals $1$.
    Since $\eta_d(a_1) = \cdots = \eta_d(a_s)$, for each $c \in \F_q^*$ the value of $\eta_d(c a_i)$ is the same for every $1 \le i \le s$. This gives us
    $$\sum_{c \in \F_q^*} \prod_{i = 1}^{s} S(ca_i, d) = \sum_{c \in \F_q^*} \left( S(c a_1, d) \right)^s.$$
    As we have seen in Lemma \ref{Lema S(c a, d)}, the value of $S(c a_1, d)$ will be $C_1(d)$ if $c a_1$ is a $d$-th power, and $C_2(d)$ if it isn't. Hence,
    $$\sum_{c \in \F_q^*} \left( S(c a_1, d) \right)^s = \frac{q-1}{d} C_1(d)^s + \frac{(q - 1)(d - 1)}{d} C_2(d)^s,$$
    which yields the result when applied in (\ref{N*(g) com S(c a_i, d_i) generico}). Other cases are analogous.
    
\end{proof}

For instance, the diagonal polynomial $g(x,y) = x^4 + y^4$ over $\F_{81}$ is such that the exponent $d = 4$ is a $(3, 2)$-admissible integer. We can compute $C_1(4) = -28$ and $C_2(4) = 8$, and use Theorem \ref{teorema diagonal caso b = 0} to conclude that the number of $*$-roots is
\begin{align*}
    N^*(g) &= \frac{1}{81}\left(80^2 + 20 (-28)^2 + 60 (8)^2 \right)\\
    &= 320.
\end{align*}

\begin{theorem}\label{teorema diagonal b ne 0}
Let $g \in \F_q[x_1, \dots, x_s]$ be a polynomial of the form (\ref{g todos ds iguais}) where $b \ne 0$. Then
$$N^*(g) = \begin{cases}
        \frac{1}{q}\left( (q - 1)^s  - C_2(d)^s + \frac{C_1(d)\left( C_1(d)^s - C_2(d)^s\right)}{d}\right), & \text{ if } \eta_d(a_1/b) = 1,\\
        \frac{1}{q}\left( (q - 1)^s  - C_2(d)^s + \frac{C_2(d)\left( C_1(d)^s - C_2(d)^s\right)}{d}\right), & \text{ if } \eta_d(a_1/b) \ne 1.\\
    \end{cases}$$
\end{theorem}
\begin{proof}
    In the equation (\ref{N*(g) com S(c a_i, d_i) generico}), we have

\[
N^*(g) = \frac{(q - 1)^s}{q} + \frac{1}{q}\sum_{c \in \F_q^*} \overline{\psi}(cb) \prod_{i = 1}^{s} S(ca_i, d).
\]

Since $\eta_d(a_1) = \cdots = \eta_d(a_s)$ implies $S(c a_1, d) = \cdots = S(c a_s, d)$, we can write

\begin{equation}\label{N^*(f) com S(c, d)}
\begin{split}
N^*(g) &= \frac{(q - 1)^s}{q} + \frac{1}{q} \sum_{c \in \F_q^*} \overline{\psi}(cb) S(ca_1, d)^s\\
&= \frac{(q - 1)^s}{q} + \frac{1}{q} \sum_{c \in \F_q^*}\overline{\psi}(c b a_1^{-1}) S(c, d)^s.
\end{split}
\end{equation}

We will assume that $p$ is odd and $(-1)^{h (p^r + 1)/d} = 1$. The cases where $p = 2$, or where $p$ is odd and $(-1)^{h (p^r + 1)/d} = -1$ are analogous. In this case, $S(c, d) = C_1(d)$ if $c$ is a $d$-th power and $C_2(d)$ otherwise. Since $\F_q^* = \{ \alpha^k: 1 \le k \le q - 1\}$, we can separate the sum $\sum_{c \in \F_q^*}\overline{\psi}(c b a_1^{-1}) S(c, d)^s$ according to whether $c$ is a $d$-th power of some element of $\F_q^*$:

    \begin{align*}
        \sum_{c \in \F_q^*}\overline{\psi}(c b a_1^{-1}) S(c, d)^s &= \sum_{l = 1}^{(q - 1)/d} \overline{\psi}(\alpha^{ld} b a_1^{-1}) C_1(d)^s + \left(\sum_{c \in \F_q^*} \overline{\psi}(c b a_1^{-1})   -  \sum_{l = 1}^{(q - 1)/d} \overline{\psi}(\alpha^{ld }b a_1^{-1})  \right) C_2(d)^s\\
        &= - C_2(d)^s + \left( C_1(d)^s - C_2(d)^s\right) \cdot \sum_{l = 1}^{(q - 1)/d} \overline{\psi}(\alpha^{ld}b a_1^{-1})\\
        &= - C_2(d)^s + \frac{\left( C_1(d)^s - C_2(d)^s\right)}{d} \cdot \sum_{x \in \F_q^*} \overline{\psi}(b a_1^{-1}x^d)\\
        &= - C_2(d)^s + \frac{\left( C_1(d)^s - C_2(d)^s\right)}{d} \cdot \sum_{x \in \F_q^*} \psi(-b a_1^{-1} x^d),\\
    \end{align*}
    \noindent where the sum $\sum_{x \in \F_q^*} \psi(- b a_1^{-1} x^d)$ is $S(-b/a_1, d)$, which equals $C_1(d)$ if $\eta_d(-b/a_1) = 1$ and $C_2(d)$ otherwise. This condition can be simplified.  Since $2r \mid m$, we have that $(q- 1)/(p^r + 1) = \sum_{j = 0}^{m/r - 1} p^{rj}$ is even, which implies $2 (p^r + 1) \mid q - 1 $. Then, the fact that $d \mid p^r + 1$ implies that $2d \mid (q - 1)$. Since $\alpha^{(q-1)/2} = -1$, it follows that $\eta_d(-1) = 1$. Therefore, the condition $\eta_d(-b/a_1) = 1$ is equivalent to $\eta_d(a_1/b) = 1$. This leads to the conclusion that,
    \begin{equation}\label{N*(f) com b ne 0}
        N^*(g) = \begin{cases}
        \frac{1}{q}\left( (q - 1)^s  - C_2(d)^s + \frac{C_1(d)\left( C_1(d)^s - C_2(d)^s\right)}{d}\right), & \text{ if } \eta_d(a_1/b) = 1,\\
        \frac{1}{q}\left( (q - 1)^s  - C_2(d)^s + \frac{C_2(d)\left( C_1(d)^s - C_2(d)^s\right)}{d}\right), & \text{ if } \eta_d(a_1/b) \ne 1.\\
    \end{cases}
    \end{equation}

\end{proof} 


For instance, let us consider the polynomial $g(x, y, z) = x^4 + y^4 + z^4 - 1 \in \F_{81} [x, y, z]$. This polynomial satisfies the condition $\eta_4(a_1/b)= \eta_4(1) = 1$, thus we use the first formula from Theorem \ref{teorema diagonal b ne 0}. Since $C_1(4) = -28$ and $C_2(4) = 8$ when $q = 81$, we have
\begin{align*}
    N^*(g) &= \frac{1}{81} \left(80^3 - 8^3 - 7 ((-28)^3 - 8^3) \right)\\
    &= 8256.
\end{align*}

As another example, consider the diagonal polynomial $g(x,y,z) = \alpha x^{17} + \alpha^{18} y^{17} - 1 \in \F_{256} [x,y,z]$, where $\alpha$ is a primitive element of $\F_{256}$ that is not a $17$-th power of any other element. The exponent $d = 17$ is $(2, 4)$-admissible, and since $\eta_{17}(a_1/b) = \eta_{17}(\alpha) \ne 1$, we use the second formula. Since $C_1(17) = 255$ and $C_2(17) = -17$ when $q = 256$, we get
\begin{align*}
    N^*(g) &= \frac{1}{256} \left(255^2 - (-17)^2 + \frac{(-17)}{17} ((255)^2 - (-17)^2) \right)\\
    &= 0.
\end{align*}
\newline
\hfill

For the next new result, we will utilize what was shown previously to determine the number of roots of full polynomials of the form in \eqref{polinomio cheio} that are $*$-equivalent to the previously analyzed family of diagonal polynomials with $(p,r)$-admissible exponents.

The number of solutions with at least one variable equal to zero is straightforward to compute: since every term in the polynomial contains every variable, if at least one variable is zero then every non-constant term vanishes. Thus when requiring at least one of the coordinates to be $0$, the number of roots is $0$ if $b \ne 0$ and $q^n - (q - 1)^n$ if $b = 0$. Hence, if we assume that a polynomial of the form \eqref{polinomio cheio} is $*$-equivalent to a polynomial $g(X)$ for which the number of $*$-roots is known, then we can calculate $N(f)$ as
\begin{equation}\label{N(f) para polinomios cheios}
    N(f) = \begin{cases}
        N^*(g), & \text{ if } b \ne 0;\\
        q^n - (q - 1)^n + N^*(g),& \text{ if } b = 0.
    \end{cases}
\end{equation}


\begin{theorem}\label{teorema principal cheio}
    Let $f$ be a full polynomial of the form (\ref{polinomio cheio}) in $\F_q[x_1, \dots, x_n]$ that is $*$-equivalent to a diagonal polynomial $g$ of the form (\ref{g todos ds iguais}) with $s$ variables, where $n \ge s$, $d \ge 3$ is $(p, r)$-admissible for some positive integer $r$, and $\eta_{d}(a_1) = \cdots = \eta_{d}(a_s)$.
    Then
    \footnotesize{
    $$N(f) = \begin{cases}
    q^n - (q - 1)^n + \frac{(q - 1)^{n - s}}{q} \left( (q - 1)^{s}+ \frac{q-1}{d} C_1(d)^s + \frac{(q - 1)(d - 1)}{d} C_2(d)^s \right),& \text{ if } b = 0,\\
     \frac{(q - 1)^{n - s}}{q} \left((q - 1)^s  - C_2(d)^s + \frac{C_1(d)\left( C_1(d)^s - C_2(d)^s\right)}{d}\right),& \text{ if } b \ne 0 \text{ and } \eta_d(a_1/b) = 1,\\
     \frac{(q - 1)^{n - s}}{q} \left((q - 1)^s  - C_2(d)^s + \frac{C_2(d)\left( C_1(d)^s - C_2(d)^s\right)}{d} \right),& \text{ if }b \ne 0 \text{ and } \eta_d(a_1/b) \ne 1,\\
    \end{cases}$$}
    \noindent where $C_1(d)$ and $C_2(d)$ are as stated in Definition \ref{definicao A e B}.
\end{theorem}
\begin{proof}    
    In this context, $g(x_1, \dots, x_s)$ is considered a polynomial in $\F_q[x_1, \dots, x_n]$, hence the $n - s$ variables $x_{s+1}, \dots, x_n$ are free and can take any value in $\F_q^*$ without affecting the value of the polynomial. Thus, the number of solutions $N_s^*(g)$ over $\F_q[x_1, \dots, x_n]$ is $(q - 1)^{n -s}$ times the number of $*$-roots over $\F_q[x_1, \dots, x_s]$ obtained from Theorems \ref{teorema diagonal caso b = 0} and \ref{teorema diagonal b ne 0}. Substituting these into (\ref{N(f) para polinomios cheios}) yields the result.
\end{proof}

Let us conclude with some examples. Let us fix $p =2$, $m = 4$ and $d = 5$ such that $q = 16$. Here $d$ is $(2,2)$-admissible. We will count the total number of roots of the polynomial $f(x, y, z) = x^6 y^2 z + x y^7 z^{11} \in \F_q[x,y,z]$, which is $*$-equivalent to $g(x,y,z) = x^5 + y^5$, since we have the following row-equivalence between the matrices $\tilde{D}_f$ and $\tilde{D}_g$: 
\begin{equation*}
    \begin{bmatrix}
        1 & 1 \\
        6 & 1 \\
        2 & 7 \\
        1 & 11\\
    \end{bmatrix}= 
    \begin{bmatrix}
        1 & 0 & 0 & 0\\
        1 & 1 & 0 & 0\\
        2 & 0 & 1 & 0\\
        1 & 0 & 2 & 1 
    \end{bmatrix}
    \begin{bmatrix}
        1 & 1 \\
        5 & 0 \\
        0 & 5 \\
        0 & 0\\
    \end{bmatrix}.
\end{equation*}
The total number of variables is $n = 3$, and the diagonal polynomial $g$ can be considered a polynomial in only $s = 2$ variables. We compute $C_1(5) = 15$ and $C_2(5) = -5$. Since the constant term is $0$, we can apply the first formula in Theorem \ref{teorema principal cheio} to obtain
\begin{align*}
    N(f) &= 
        16^3 - (15)^3 + \frac{15}{16}\left( 15^2 +  \frac{15}{5} 15^2 + \frac{15 \cdot 4}{5} (-5)^2\right)\\
        &= 1846.
\end{align*}
As another example, let us fix $p = 3$, $m = 6$, $d =7$, and $\alpha$ as a primitive element of $\F_q$ that is not a $7$-th power of any other element.  We will count the total number of roots of the polynomial $f(x,y) = x^7 + 2 x^7 y^{21} - \alpha \in \F_{729}[x,y]$, which is $*$-equivalent to $g(x,y) = x^7 + 2 y^{7} - \alpha$ since they have the same coefficient vector and $\tilde{D}_f$ and $\tilde{D}_g$ are row-equivalent:
\begin{equation*}
    \begin{bmatrix}
        1 & 1 & 1\\
        7 & 7 & 0\\
        0 & 21 & 0\\
    \end{bmatrix}= 
    \begin{bmatrix}
        1 & 0 & 0\\
        0 & 1 & 1\\
        0 & 0 & 3\\
    \end{bmatrix}
    \begin{bmatrix}
        1 & 1 & 1\\
        7 & 0 & 0\\
        0 & 7 & 0\\
    \end{bmatrix}.
\end{equation*}
The number of variables is the same for both polynomials, $n = s = 2$. Since $7$ is $(3,3)-$admissible, we compute $C_1(7) = 161$ and $C_2(7) = - 28$. Since $\eta_7(a_1/b) = \eta_7(\alpha) \ne 1$, we use the third formula in Theorem \ref{teorema principal cheio} to obtain
\begin{align*}
    N(f) &= 
        \frac{728^0}{729} \left( 728^2 - (-28)^2 + \frac{(-28)(161^2 - (-28)^2)}{7} \right) \\
        &= 588.
\end{align*}

\section*{Acknowledgements}

The author was supported in part by the Coordena\c{c}\~ao de Aperfei\c{c}oamento de Pessoal de N\'ivel Superior-Brazil (CAPES) - Finance Code 001.

\end{document}